\documentclass[a4paper,12pt]{amsart}

\usepackage{amsmath,amssymb}
\usepackage{ifthen}
\usepackage{graphicx}
\usepackage[T1]{fontenc}
\usepackage[utf8]{inputenc}
\usepackage[usenames,dvipsnames]{color}
\usepackage[english]{babel}
\usepackage{fancyhdr}
\usepackage{fancybox}
\usepackage{tikz}
\usepackage{cite}
\usepackage[backref=page, colorlinks=true, linkcolor=blue, citecolor=red, urlcolor=blue]{hyperref}

\renewcommand*{\backref}[1]{}
\renewcommand*{\backrefalt}[4]{%
	\ifcase #1 %
	(Not cited.)%
	\or
	[#2]%
	\else
	[#2]%
	\fi
}

\nonstopmode 
\numberwithin{equation}{section}
\theoremstyle{plain}
\newtheorem{prop}{Proposition}

\newtheorem{conj}{Conjecture}

\theoremstyle{definition}
\newtheorem{defi}{Definition}[section]
\newtheorem{exm}{Example}[section]
\newtheorem{cor}{Corollary}[section]
\newtheorem{thm}{Theorem}[section]

\newtheorem{lem}{Lemma}[section]
\newtheorem{prob}{Problem}
\newtheorem{rem}{Remark}[section]

\theoremstyle{plain}

\newtheorem*{thmA}{Theorem A}

\newtheorem*{lemA}{Lemma A}

\newcounter{minutes}
\divide\time by 60
\newcounter{hours}
\multiply\time by 60
\addtocounter{minutes}{-\time}

\newcounter {own}
\def\theown {\thesection.\arabic{own}}

\newenvironment{pf}[1][]{%
	\vskip 3mm
	\noindent
	\ifthenelse{\equal{#1}{}}%
	{{\slshape Proof. }}%
	{{\slshape #1.} }%
}%
{\qed\bigskip}

\newcounter{alphabet}

\def\be{\begin{equation}}
	\def\ee{\end{equation}}

\newcommand{\bee}{\begin{enumerate}}
	\newcommand{\eee}{\end{enumerate}}

\newcommand{\blem}{\begin{lem}}
	\newcommand{\elem}{\end{lem}}
\newcommand{\bthm}{\begin{thm}}
	\newcommand{\ethm}{\end{thm}}
\newcommand{\bcor}{\begin{cor}}
	\newcommand{\ecor}{\end{cor}}
\newcommand{\beg}{\begin{examp}}
	\newcommand{\eeg}{\end{examp}}
\newcommand{\begs}{\begin{examples}}
	\newcommand{\eegs}{\end{examples}}
\newcommand{\bdefn}{\begin{defn}}
	\newcommand{\edefn}{\edefn}
	\newcommand{\bprob}{\begin{prob}}
		\newcommand{\eprob}{\end{prob}}
	\newcommand{\bei}{\begin{itemize}}
		\newcommand{\eei}{\end{itemize}}
	\newcommand{\bcon}{\begin{conj}}
		\newcommand{\econ}{\end{conj}}
	\newcommand{\bprop}{\begin{prop}}
		\newcommand{\eprop}{\eprop}
		\newcommand{\br}{\begin{rem}}
			\newcommand{\er}{\end{rem}}
		\newcommand{\bpf}{\begin{pf}}
			\newcommand{\epf}{\end{pf}}
		\newcommand{\ba}{\begin{array}}
			\newcommand{\ea}{\end{array}}
		\newcommand{\beq}{\begin{eqnarray}}
			\newcommand{\beqq}{\begin{eqnarray*}}
				\newcommand{\eeq}{\end{eqnarray}}
			\newcommand{\eeqq}{\end{eqnarray*}}

\begin{document}

\title{Schwarzian norm estimates for some classes of analytic and harmonic mappings}

\author{Molla Basir Ahamed}
\address{Molla Basir Ahamed, Department of Mathematics, Jadavpur University, Kolkata-700032, West Bengal, India.}
\email{mbahamed.math@jadavpuruniversity.in}

\author{Rajesh Hossain}
\address{Rajesh Hossain, Department of Mathematics, Jadavpur University, Kolkata-700032, West Bengal, India.}
\email{rajesh1998hossain@gmail.com}

\subjclass[{AMS} Subject Classification:]{Primary: 30C45; 30C55; 30C80}
\def\thefootnote{}
\footnotetext{ {\tiny File:~\jobname.tex,
printed: \number\year-\number\month-\number\day,
          \thehours.\ifnum\theminutes<10{0}\fi\theminutes }
} \makeatletter\def\thefootnote{\@arabic\c@footnote}\makeatother
\begin{abstract} 
	Let $\mathcal{A}$ be the normalized class of analytic functions $f$ in the unit disc $\mathbb{D} := \{z \in \mathbb{C} : \vert{}z\vert{} < 1\}$. For $\beta > 1$, let $\mathcal{N}(\beta)$ denote the subclass of $\mathcal{A}$ satisfying $\text{Re}\{1 + z f''(z)/f'(z)\} < \beta$ for $z \in \mathbb{D}$. The main purpose of this paper is to establish sharp bounds for the pre-Schwarzian norm $\Vert{}P_f\Vert{}$ and Schwarzian norm $\Vert{}S_f\Vert{}$ for functions $f \in \mathcal{N}(\beta)$, parametrized by $f''(0)$, with special emphasis on the case $f''(0) = 0$. In addition, sharp growth, distortion, and radius results (convexity and concavity) for $\mathcal{N}(\beta)$ are obtained. As an application, we determine the sharp pre-Schwarzian norm estimate for harmonic mappings $f = h + \bar{g}$ whose analytic part $h$ belongs to $\mathcal{N}(\beta)$.
\end{abstract}
\maketitle
\pagestyle{myheadings}
\markboth{M. B. Ahamed and R. Hossain}{Pre-Schwarzian and Schwarzian norm estimates}
\tableofcontents
\section{\bf Introduction}
Let \( \mathbb{D} = \{z \in \mathbb{C} : |z| < 1\} \) be the unit disk, and define \( \mathcal{H} \) as the class of analytic functions on \( \mathbb{D} \). The subclass \( \mathcal{LU} \) consists of locally univalent functions, \textit{i.e.,} functions \( f \in \mathcal{H} \) with \( f^{\prime}(z) \neq 0 \) for all \( z \in \mathbb{D} \). Let $\mathcal{A}$ denote the subclass of $\mathcal{H}$, a class of analytic functions on the unit disk $\mathbb{D}$,  consisting of functions $f$ with normalized conditions $f(0)=f^{\prime}(z)-1=0$. Thus, any function $f$ in $\mathcal{A}$ has the Taylor series expansion of the form 
\begin{align}\label{Eq-1.1}
	f(z)=z+\sum_{n=2}^{\infty}a_nz^n\;\; \mbox{for all}\;\; z\in\mathbb{D}.
\end{align}
Let $\mathcal{S}$ be the subclass of $\mathcal{A}$ consisting of univalent (that is, one-to-one) functions. A function $f\in \mathcal{A}$ is called starlike (with respect to the origin) if $f(\mathbb{D})$ is starlike with respect to the origin and convex if $f(\mathbb{D})$ is convex. For a locally univalent analytic function $f(z)$ defined in a simply connected domain ${\Omega}$, the pre-Schwarzian derivative $P_f$ and the Schwarzian derivative $S_f$ are defined by
\begin{align*}
	P_f=\frac{f^{\prime\prime}(z)}{f^{\prime}(z)}\;\mbox{and}\;S_f=(P_f)^{\prime}(z)-\frac{1}{2}(P_f)^2(z)=\frac{f^{\prime\prime\prime}(z)}{f^{\prime\prime}(z)}-\frac{3}{2}\left(\frac{f^{\prime\prime}(z)}{f^{\prime}(z)}\right)^2,
\end{align*}
respectively. The pre-Schwarzian and Schwarzian norms are defined by
\begin{align*}
	||P_f||_{\Omega}= \sup_{z\in{\Omega}}|P_f|\eta^{-1}_{\Omega}\;\mbox{and}\;||S_f||_{\Omega}= \sup_{z\in{\Omega}}|S_f|\eta^{-2}_{\Omega}
\end{align*}
respectively, where $\eta_{\Omega}$ is the Poincare density. In particular ${\Omega}=\mathbb{D}$, then $||S_f||_{\Omega}$ and $||P_f||_{\Omega}$ are denoted by $||Sh||$ and $||Ph||$,  respectively.\vspace{2mm}

The pre-Schwarzian and Schwarzian derivatives are key tools in geometric function theory, particularly for characterizing Teichmüller space through embedding models. They also play a crucial role in studying the inner radius of univalency for planar domains and quasiconformal extensions \cite{Lehto-1987,Lehto-JAM-1979}. Their study dates back to Kummer (1836), who introduced the Schwarzian derivative in the context of hypergeometric PDEs. Since then, extensive research has explored their connections to univalent functions, leading to several sufficient conditions for univalency.\vspace{2mm}

It is well-known that the pre-Schwarzian norm $||P_f||\leq 6$ holds for the univalent analytic function $f$ is defined in $\mathbb{D}$. In $1972$, Becker \cite{Becker-JRAM-1983} used the pre-Schwarzian derivative to obtain the sufficient condition that the function in $\mathbb{D}$ is univalent, in other words, if $||P_f||\leq1$, then the function $f$ is univalent in $\mathbb{D}$. In $1976$, Yamashita \cite{Yamashita-MM-1976} proved that $||P_f||$ is finite if, and only if, $f$ is uniformly locally univalent in $\mathbb{D}$, \emph{i.e.}, there exists a constant $\rho$ such that $f$ is univalent on the hyperbolic disk $|(z-a)/(1-\bar{a}z)|<\tanh\rho$ of radius $\rho$ for every $a\in\mathbb{D}$. Sugawa \cite{Sugawa-AUMCDS-1996} studied the strongly starlike functions of order $\alpha\; (0<\alpha\leq1)$. Yamashita\cite{Yamashita-HMJ-1999} generalized sugawa's results by a general class named Gelfer-starlike of exponential order $\alpha (\alpha>0)$ and the Gelfer-close-to-convex of exponential order $(\alpha,\beta)$ ($\alpha>0$, $\beta>0$). These Gelfer classes also contain the classical starlike, convex, close-to-convex all of order $\alpha$ ($0\leq\alpha<1$), which are denote by $\mathcal{S^*(\alpha)}$, $\mathcal{C(\alpha)}$, $\mathcal{K(\alpha)}$ respectively, and so on.\vspace{2mm}

The pioneering work on the bound \( ||S_f|| \leq 6 \) for a univalent function \( f \in \mathcal{A} \) was first introduced by Kraus \cite{Kraus-1932} and later revisited by Nehari \cite{Nehari-BAMS-1949}. In the same paper, Nehari also proved that if $||S_f||\leq2$, then the function $f$ is univalent in $\mathbb{D}$.  For recent development of the pre-Schwarzian norm estimates of other function forms such as convolution operator and integral operator, we refer to the articles \cite{Choi-Kim-Ponnusamy-Sugawa-JMAP-2005,Kim-Sugawa-PEMS-2006,Parvatham-Ponnusamy-Sahoo-HMJ-2008,Ponnusamy-Sahoo-JMAA-2008,Ponnusamy-Sugawa-JKMS-2008,Ponnusamy-Sahoo-M-2008,Kanas-AMC-2009} and references therein.\vspace{2mm}

The Schwarzian norm plays a significant role in the theory of quasiconformal mappings and Teichm\"uller space (see \cite{Lehto-1987}). A mapping \( f : \hat{\mathbb{C}} \to \hat{\mathbb{C}} \) of the Riemann sphere \( \hat{\mathbb{C}} := \mathbb{C} \cup \{\infty\} \) is said to be a \( k \)-quasiconformal (\( 0 \leq k < 1 \)) mapping if it is a sense-preserving homeomorphism of \( \hat{\mathbb{C}} \) and has locally integrable partial derivatives on \( \mathbb{C} \setminus \{f^{-1}(\infty)\} \), satisfying \( |f_{\bar{z}}| \leq k |f_z| \) almost everywhere.  On the other hand, Teichmüller space \( \mathcal{T} \) can be identified with the set of Schwarzian derivatives of analytic and univalent functions on \( \mathbb{D} \) that have quasiconformal extensions to \( \hat{\mathbb{C}} \). It is known that \( \mathcal{T} \) is a bounded domain in the Banach space of analytic functions on \( \mathbb{D} \) with a finite hyperbolic sup-norm (see \cite{Lehto-1987}).  \vspace{2mm}

The Schwarzian derivative and quasiconformal mappings are connected through key results presented below.
\begin{thmA}\emph{\cite{Ahlfrors-Weill-PAMS-2012,Kühnau-MN-1971}}
	If $f$ extends to a $k$-quasiconformal $(0\leq k<1)$ mapping of the Riemann share $\hat{\mathbb{C}}$, then $||S_f||\leq 6k$. Conversely, if $||S_f||\leq 2k$, then $f$ extends to a $k$-quasiconformal mapping of the Riemann sphere $\hat{\mathbb{C}}$.
\end{thmA}
Regarding to the estimates of the Schwarzian norm for the subclasses of univalent functions \emph{i.e.,} of functions $f$ that satisfy:
\begin{align*}
	\bigg|\arg\left(\frac{zf^{\prime}(z)}{f(z)}\right)\bigg|<\alpha\frac{\pi}{2},\; z\in\mathbb{D},
\end{align*}
where $0\leq \alpha<1$. Fait \textit{et al.} \cite{Fait-Krzy-Zygmunt-CMH-1976} studied the strong starlike function. In $1996$, Suita \cite{Suita-JHUED-1996} studied the class $\mathcal{C(\alpha)}$, $0\leq \alpha<1$ and using the integral representation of functions in $\mathcal{C}(\alpha)$ proved that the Schwarzian norm satisfies the sharp inequality 
\begin{align*}
	||S_f||\leq\begin{cases}
		2,\;\;\;\;\;\;\;\;\;\;\;\;\;\;\;\;\; \mbox{if}\; 0\leq \alpha\leq 1/2,\\
		8\alpha(1-\alpha),\;\;\;\;\mbox{if}\; 1/2<\alpha<1.
	\end{cases}
\end{align*}
For a constant $\beta\in (-\pi/2, \pi/2)$, a function $f\in\mathcal{A}$ is called $\beta$-spiral like if $f$ is univalent on $\mathbb{D}$ and for any $z\in\mathbb{D}$, the $\beta$-logarithmic spiral $\{f(z)\exp\left(-e^{i\beta}t\right);\; t\geq 0\}$ is contained in $f(\mathbb{D})$. It is equivalent to the condition that ${\rm Re} \left(e^{-i\beta}zf^{\prime}(z)/f(z)\right)>0$ in $\mathbb{D}$ and we denote by $\mathcal{SP}(\beta)$, the set of all $\beta$-spiral like functions. Okuyama \cite{Okuyama-CVTA-2000} give the best possible estimate of the norm of pre-Schwarzian derivatives for the class $\mathcal{SP}(\beta)$.\vspace{2mm}

A function $f\in\mathcal{A}$ is said to be uniformly convex function if every circular arc (positively oriented) of the form $\{z\in\mathbb{D} : |z-\eta|=r\}$, $\eta\in\mathbb{D}$, $0<r<|\eta|+1$ is mapped by $f$ univalently onto a convex arc. The class of all uniformly convex functions is denoted by $\mathcal{UCV}$. In particular, $\mathcal{UCV}\subset \mathcal{K}$. It is well-known that (see \cite{Goodman-APM-1991}) a function $f\in\mathcal{A}$ is uniformly convex if, and only if, 
\begin{align*}
	{\rm Re}\left(1+\frac{z f^{\prime\prime}(z)}{f^{\prime}(z)}\right)>\bigg|\frac{z f^{\prime\prime}(z)}{f^{\prime}(z)}\bigg|^2\; \mbox{for}\; z\in\mathbb{D}.
\end{align*}
In \cite{Kanas-Sugawa-APM-2011}, Kanas and Sugawa established that the Schwarzian norm satisfies \( ||S_f|| \leq 8/\pi^2 \) for all \( f \in \mathcal{UCV} \), with the bound being sharp.  Recently, Schwarzian norm estimates for other subclasses of univalent functions have been gradually studied by many people, such as concave function class \cite{Bhowmik-Wriths-CM-2012}. Therefore, by using the pre-Schwarzian  and Schwarzian norms to study the univalence and quasiconformal extension problems of analytic function arouse a new wave of research interest.\vspace{2mm}

Let $\mathcal{S}^*(\alpha)$ and $\mathcal{C}(\alpha)$ denote respectively, the classes of starlike and convex functions of order $\alpha$ for $0\leq \alpha<1$ in $\mathcal{S}$. It is well-known that a function $f\in\mathcal{A}$ belongs to $\mathcal{S}^*(\alpha)$ if, and only if, ${\rm Re}(zf^{\prime}(z)/f(z))>\alpha$ for $z\in\mathbb{D}$, and $f\in\mathcal{C}(\alpha)$ if, and only if, ${\rm Re}(1+zf^{\prime\prime}(z)/f^{\prime}(z))>\alpha$. Similarly, a function $f\in\mathcal{A}$ belongs to $\mathcal{K}$, the class of close-to-convex functions, if and only if, there exists $g\in\mathcal{S}^*$ such that ${\rm Re}[e^{i\tau}(zf^{\prime}(z))/g(z)]>0$ for $z\in\mathbb{D}$ and $\tau\in (-\pi/2, \pi/2)$. Thus, it is easy to see that  $\mathcal{C}\subset\mathcal{S}^*\subset\mathcal{K}\subset\mathcal{S}$. In particular, when $\tau=0$, then the resulting subclass of the close-to-convex functions is denoted by $\mathcal{K}_0$.\vspace{2mm}

In this article, we examine two distinct function classes for fixed $\beta>1$, $\mathcal{M}(\beta)$ and $\mathcal{N}(\beta)$ \cite{Ali-Allu-JAMS-2016}, defined by
\begin{align*}
	\mathcal{M}(\beta)&=\bigg\{f\in\mathcal{A}:\;{\rm Re}\;\bigg\{\frac{zf^{\prime}(z)}{f(z)}\bigg\}<\beta\;\;\mbox{for}\;z\in\mathbb{D}\bigg\},
\end{align*}
and
\begin{align*}
	\mathcal{N}(\beta)=\bigg\{f\in\mathcal{A}:\;{\rm Re}\;\bigg\{1+\frac{zf^{\prime\prime}(z)}{f^{\prime}(z)}\bigg\}<\beta\;\;\mbox{for}\;z\in\mathbb{D}\bigg\}.
\end{align*}
respectively. It is easy to see that $f\in\mathcal{N}(\beta)$ if, and only if, $zf^{\prime}\in\mathcal{M}(\beta)$. In $1941$, Ozaki \cite{Ozaki-SRTBD-1941} introduced the class $\mathcal{N}(3/2)$ and prove the functions in $\mathcal{N}(3/2)$ are univalent in $\mathbb{D}$. Moreover, functions in the class $\mathcal{N}(3/2)$ were proved to be starlike in the unit disk $\mathbb{D}$ \cite{                Ponnuswamy-Rajasekaran-SJM-1995,Jovanovi´c-Obradovi´c-F-1995}. Thus the class $\mathcal{N}(\beta)$ is included in the class $\mathcal{S^*}$ for $1<\beta\leq3/2$. Also, we note that functions in the class $\mathcal{N}(\beta)$ need not be univalent in the unit disk $\mathbb{D}$ if $\beta>3/2$. For $1<\beta\leq4/3$, the class $\mathcal{M}(\beta)$ was introduced by \cite{Ali-Allu-JAMS-2016}. Later, full classes were investigated by Owa and Nishiwaki \cite{Nishiwaki-Owa-IJMMS-2002,Owa-Nishiwaki-IJMMS-2002} and also by Owa and Srivastava \cite{Owa-Srivastava-JIPAM-2002}. Recently, Obradovic \cite{Obradovi´c-Ponnusamy-Wirths-SMJ-2013} studies the class $\mathcal{N}(\beta)$ for $1<\beta\leq3/2$.\vspace*{2mm}

Before we state our main result, let us recall another important and useful tool known as the differential subordination technique. Many problems in geometric function theory can be solved in a simple and sharp manner with the help of differential subordination.  A function $f\in\mathcal{H}$ is said to be subordinate to another function $g\in\mathcal{H}$ if there exists an analytic function $\omega : \mathbb{D}\to\mathbb{D}$ with $\omega(0)=0$ such that $f(z)=g(\omega(z))$ and it is denoted by $f\prec g$. Moreover, when $g$ is univalent, then $f\prec g$ if, and only if, $f(0)=g(0)$ and $f(\mathbb{D}\subset g(\mathbb{D})$.\vspace{2mm} 

In terms of subordination, the class $\mathcal{M}(\beta)$ and $\mathcal{N}(\beta)$ can be defined as:\
\begin{align}\label{Eq-1.2}
	f\in\mathcal{M}(\beta)\; \mbox{if, and only if,}\;{\rm Re  }\left(\frac{zf^{\prime}(z)}{f(z)}\right)\prec\frac{1+(1-2\beta)z}{1-z},
\end{align}
and
\begin{align}\label{Eq-1.3}
	f\in\mathcal{N}(\beta)\; \mbox{if, and only if,}\;{\rm Re  }\left(\frac{zf^{\prime\prime}(z)}{f^{\prime}(z)}\right)\prec-\frac{2(\beta-1)z}{1-z}.
\end{align} In \cite{Chuaqui-Duren-Osgood-AASFM-2011}, Chuaqui \emph{et. al.} proved a result by applying the Schwarz-Pick lemma and the fact that the expression $1+z(f^{\prime\prime}/f^{\prime})(z)$ is subordinate to the half-plan mapping $\ell(z)=(1+z)/(1-z)$, which is 
\begin{align}\label{Eq-1.4}
	1+\frac{zf^{\prime\prime}(z)}{f^{\prime}(z)}=\ell(w(z))=\frac{1+w(z)}{1-w(z)}
\end{align}
for some function $w : \mathbb{D}\to\mathbb{D}$ holomorphic and such that $w(0)=0$. \vspace{2mm}

The expression which is defined in \eqref{Eq-2.2} allowed us to obtain other characterizations for the convex functions
\begin{align}\label{Eq-22.3}
	f\in\mathcal{C}\; \mbox{if, and only if,}\; {\rm Re}\left(1+\frac{zf^{\prime\prime}(z)}{f^{\prime}(z)} \right)\geq \frac{1}{4}\left( 1-|z|^2\right)\bigg| \frac{f^{\prime\prime}(z)}{f^{\prime}(z)}  \bigg|^2,
\end{align}
and 
\begin{align}\label{Eq-22.4}
	f\in\mathcal{C}\; \mbox{if, and only if,}\; \bigg|\left(1-|z|^2\right)\frac{f^{\prime\prime}(z)}{f^{\prime}(z)} -2\bar{z} \bigg|\leq 2,
\end{align}
for all $z\in\mathbb{D}$.\vspace{2mm}

The geometric properties of analytic functions such as growth, distortion, and the behavior of the Schwarzian and pre-Schwarzian derivatives are central to geometric function theory. Motivated by the works of Hernández and Martín (see \cite{Carrasco-Hernández-AMP-2011}), Wang \emph{et. al.} (see \cite{Wang-Li-Fan-MM-2024}), and more recently Ahamed \emph{et al.} (see \cite{Ahamed-Allu-Hossain-MM-2025,Ahamed-Hossain-MS-2026,Ahamed-Hossain-Wang-LJM-2026}), who studied these aspects for various function classes, we explore similar properties for new subclasses. This paper focuses on sharp norm estimates and bounds, culminating in a final theorem that offers a significant result with potential applications in future research.\vspace{2mm}

The paper is organized as follows. Section \ref{Sec-2} presents key properties of the class $\mathcal{N}(\beta)$, including sharp bounds for the pre-Schwarzian and Schwarzian derivative norms, as well as growth and distortion theorems. In Section \ref{Sec-3}, we evaluate the sharp pre-Schwarzian norm for harmonic mappings with a fixed analytic part in $\mathcal{N}(\beta)$. Section \ref{Sec-4} is devoted to establishing the Bloch constant for functions in this class. Finally, Section \ref{Sec-5} determines the sharp radii of convexity and concavity for $\mathcal{N}(\beta)$.
\section{\bf{Pre-Schwarzian and Schwarzian norm Estimates for the class $\mathcal{N}(\beta)$}}\label{Sec-2}
In this section, we first establish an equivalent characterization for the class $\mathcal{N}(\beta)$ which leads to Theorem \ref{Th-2.1}. We then present the distortion and growth theorem (Theorem \ref{Th-2.2}), followed by the derivation of the pre-Schwarzian and Schwarzian norms for functions in this class, expressed in terms of \( f''(0) \).
\begin{thm}\label{Th-2.1} For $\beta>1$, the following are equivalent:
	\begin{enumerate}
		\item[(i)] $f\in\mathcal{N}(\beta)$.\vspace{2mm}
		
		\item[(ii)] $\displaystyle{\rm Re  }\left(1+\frac{zf^{\prime\prime}(z)}{f^{\prime}(z)}\right)\leq\beta-\frac{1}{4}\left(\frac{1-|z|^2}{(\beta-1)}\right)\bigg|\frac{zf^{\prime\prime}(z)}{f^{\prime}(z)}\bigg|^2$\;\;\;\mbox{for}\;\;$\beta>1$.\vspace{2mm}
		
		\item[(iii)] $\displaystyle\bigg|	(1-|z|^2)\left(\frac{f^{\prime\prime}(z)}{f^{\prime}(z)}\right)+2(\beta-1)\bar{z}\bigg|\leq2(\beta-1)
		$.
	\end{enumerate}The inequalities \emph{(ii)} and \emph{(iii)} both are sharp for the function
	\begin{align}\label{Eq-22.11}
	f^{\prime}_{\lambda}(z)=(1-z)^{2(\beta-1)}\; \mbox{for}\; z\in\mathbb{D}\;\mbox{with}\;\beta\in(1,3/2].
	\end{align}
\end{thm}
As a consequence of Theorem \ref{Th-2.1}, we obtain two results for functions in the class $\mathcal{N}(3/2)$ involving inequalities.
\begin{cor}\label{Cor-2.1A}
	If $f\in \mathcal{N}(3/2)$, then we have the following inequality
	\begin{align}\label{Eq-2.2A}
		{\rm Re  }\left(1+\frac{zf^{\prime\prime}(z)}{f^{\prime}(z)}\right)\leq\frac{3}{2}-\left(\frac{1-|z|^2}{2}\right)\bigg|\frac{zf^{\prime\prime}(z)}{f^{\prime}(z)}\bigg|^2\;\mbox{for}\;\beta>1.
	\end{align}
	the inequality is sharp.
	\end{cor}
	\begin{cor}\label{Cor-2.2}
		 If $f\in \mathcal{N}(3/2)$, then we have the following inequality
		\begin{align}\label{Eq-2.3A}
			\bigg|	(1-|z|^2)\left(\frac{f^{\prime\prime}(z)}{f^{\prime}(z)}\right)+\bar{z}\bigg|\leq1
		\end{align}
		the inequality is sharp.
	\end{cor}
	\begin{exm}
		For the sharpness of the inequalities \eqref{Eq-2.2A} and \eqref{Eq-2.3A}, we consider the function defined in \eqref{Eq-22.11} with $\beta=3/2$ as 
		\begin{align*}
			f^{\prime}_{3/2}(z)=1-z
		\end{align*}
		A simple computation using \eqref{Eq-22.11} shows that
		\begin{align*}
			1+\frac{zf_{3/2}^{\prime\prime}(z)}{f_{3/2}^{\prime}(z)}=1-\frac{z}{1-z}.
		\end{align*}
		Moreover, it is easy to see that
		\begin{align*}
			{\rm 	Re}\left(1+\frac{zf_{3/2}^{\prime\prime}(z)}{f_{3/2}^{\prime}(z)}\right)>0,
		\end{align*}
		hence, it is clear that $f_{3/2}\in\mathcal{N}(3/2)$.\vspace{2mm}
		
		To show the inequality \eqref{Eq-2.3A} of Corollary \ref{Cor-2.2} is sharp, we consider $z=r<1$ and establish that 
		\begin{align*}
			\bigg|	(1-|z|^2)\left(\frac{f_{3/2}^{\prime\prime}(z)}{f_{3/2}^{\prime}(z)}\right)+\bar{z}\bigg|=\bigg|	(1-r^2)\left(\frac{-1}{1-r}\right)+r\bigg|=|-1-r+r|=1.
		\end{align*}
		
		\noindent To show the inequality \eqref{Eq-2.2A} in Corollary \ref{Cor-2.1A} is sharp, we see from \eqref{Eq-2.2} (Proof of Theorem \ref{Th-2.1}) that 
		\begin{align}\label{Eq-22.44}
			\phi(z)=\frac{\frac{f_{3/2}^{\prime\prime}(z)}{f_{3/2}^{\prime}(z)}}{\frac{zf_{3/2}^{\prime\prime}(z)}{f_{3/2}^{\prime}(z)}-1}.
		\end{align}
		Thus, it is clear that $|\phi(z)|^2=1$ which further leads to 
		\begin{align*}
			{\rm Re  }\left(1+\frac{zf_{3/2}^{\prime\prime}(z)}{f_{3/2}^{\prime}(z)}\right)=\frac{3}{2}-\left(\frac{1-|z|^2}{2}\right)\bigg|\frac{zf_{3/2}^{\prime\prime}(z)}{f_{3/2}^{\prime}(z)}\bigg|^2.
		\end{align*}
	\end{exm}
	
\begin{proof}[\bf Proof of Theorem \ref{Th-2.1}]
	First, we will prove that $(i)$ is equivalent to $(ii)$. For $\beta>1$, let  $f\in\mathcal{N}(\beta)$ be of the form \eqref{Eq-1.1}. Then from \eqref{Eq-1.3}, we have
	\begin{align*}
		\frac{zf^{\prime\prime}(z)} {f^{\prime}(z)}\prec-\frac{2(\beta-1)z}{1-z},
	\end{align*}
	
	\noindent then there exists an analytic function $\omega: \mathbb{D}\rightarrow\mathbb{D}$ with $\omega(0)=0$ such that
	\begin{align}\label{Eq-2.1}
		\frac{zf^{\prime\prime}(z)} {f^{\prime}(z)}=-\frac{2(\beta-1)\omega(z)}{1-\omega(z)}.
	\end{align}
	 Let $\omega(z)=z\phi(z)$ for some analytic function $\phi$ that satisfy $\phi(\mathbb{D})\subseteq\mathbb{D}$. From \eqref{Eq-2.1}, it follows that
	\begin{align*}
		\frac{f^{\prime\prime}(z)}{f^{\prime}(z)}=-\frac{2(\beta-1)z\phi(z)}{z(1-z\phi(z))}
	\end{align*}
	which yields that
	\begin{align}\label{Eq-2.2}
		\phi(z)=\frac{\frac{f^{\prime\prime}(z)}{f^{\prime}(z)}}{\frac{zf^{\prime\prime}(z)}{f^{\prime}(z)}-2(\beta-1)}.
	\end{align}
	Since $|\phi(z)|^2\leq1$, an easy computation leads to
	\begin{align}\label{Eq-2.3}
		\bigg|\frac{f^{\prime\prime}(z)}{f^{\prime}(z)}\bigg|^2\leq4\left(\beta-1\right)^2-4\left(\beta-1\right)	{\rm Re}\left(\frac{zf^{\prime\prime}(z)}{f^{\prime}(z)}\right)+|z|^2\bigg|\frac{f^{\prime\prime}(z)}{f^{\prime}(z)}\bigg|^2.
	\end{align}
	By factorizing, we obtain
	\begin{align}
		\nonumber4(\beta-1)\left((\beta-1)-2{\rm Re}\left(\frac{zf^{\prime\prime}(z)}{f^{\prime}(z)}\right)\right)\ge(1-|z|^2)\bigg|\frac{f^{\prime\prime}(z)}{f^{\prime}(z)}\bigg|^2
	\end{align}
	which implies that
	\begin{align}
		{\rm Re}\left(\frac{zf^{\prime\prime}(z)}{f^{\prime}(z)}\right)\leq\beta-1-\left(\frac{1-|z|^2}{4(\beta-1)}\right)\bigg|\frac{f^{\prime\prime}(z)}{f^{\prime}(z)}\bigg|^2.
	\end{align}
	Consequently, we have
	\begin{align*}
		{\rm Re  }\left(1+\frac{zf^{\prime\prime}(z)}{f^{\prime}(z)}\right)\leq\beta-\left(\frac{1-|z|^2}{4(\beta-1)}\right)\bigg|\frac{f^{\prime\prime}(z)}{f^{\prime}(z)}\bigg|^2.
	\end{align*}
	\noindent Next, we will prove that $(ii)$ is equivalent to $(iii)$. Multiplying \eqref{Eq-2.3} by $(1-|z|^2)$ both side, we have
	\begin{align*}
		(1-|z|^2)^2\bigg|\frac{f^{\prime\prime}(z)}{f^{\prime}(z)}\bigg|^2\leq4\left(\beta-1\right)^2(1-|z|^2)-4\left(\beta-1\right)(1-|z|^2)	{\rm Re}\left(\frac{zf^{\prime\prime}(z)}{f^{\prime}(z)}\right)
	\end{align*}
	which implies that 
	\begin{align*}
		(1-|z|^2)^2\bigg|\frac{f^{\prime\prime}(z)}{f^{\prime}(z)}\bigg|^2+4\left(\beta-1\right)(1-|z|^2)	{\rm Re}\left(\frac{zf^{\prime\prime}(z)}{f^{\prime}(z)}\right)+4\left(\beta-1\right)^2|z|^2\leq4\left(\beta-1\right)^2.
	\end{align*}
	Thus, we have
	\begin{align*}
		\bigg|	(1-|z|^2)\left(\frac{f^{\prime\prime}(z)}{f^{\prime}(z)}\right)+2(\beta-1)\bar{z}\bigg|\leq2(\beta-1).
	\end{align*}
	This completes the proof.
\end{proof}
We obtain the following result which is distortion theorem and growth theorem for functions in the class $\mathcal{N}(\beta)$.
\begin{thm}\label{Th-2.2}
For $\beta>1$, let $f\in \mathcal{N}(\beta)$ be of the form \eqref{Eq-1.1}. Then we have
	\begin{align}
		\frac{1}{(1+|z|^2)^{\beta-1}}\leq|f^{\prime}(z)|\leq\frac{1}{(1-|z|^2)^{\beta-1}}	
	\end{align}
	and
	\begin{align}
		\int_{0}^{|z|}\frac{1}{(1+\xi^2)^{\beta-1}} d|\xi|\leq	|f(z)|\leq\int_{0}^{|z|}\frac{1}{(1-\xi^2)^{\beta-1}} d|\xi|.
	\end{align}
	All of these estimates are sharp. Equality holds for the function 
	\begin{align*}
		f_{\beta, \lambda}(z)=\int_{0}^{|z|}\frac{1}{(1-\lambda \zeta^2)^{\beta-1}}d|\zeta|
	\end{align*} 
	for some $\lambda\in\mathbb{C}$ with $|\lambda|=1$.
\end{thm}
\noindent The following result for the class \( \mathcal{N}(3/2) \) is a direct consequence of Theorem \ref{Th-2.2}

\begin{cor}
	If $f\in \mathcal{N}(3/2)$ be of the form \eqref{Eq-1.1}, then the sharp inequality 
	\begin{align}
		\frac{1}{(1+|z|^2)^{1/2}}\leq|f^{\prime}(z)|\leq\frac{1}{(1-|z|^2)^{1/2}}	
	\end{align}
	and
	\begin{align}
		\int_{0}^{|z|}\frac{1}{(1+\xi^2)^{1/2}} d|\xi|\leq	|f(z)|\leq\int_{0}^{|z|}\frac{1}{(1-\xi^2)^{1/2}} d|\xi|.
	\end{align}
	All of these estimates are sharp. Equality holds for the function 
	\begin{align*}
		f_{\beta, \lambda}(z)=\int_{0}^{|z|}\frac{1}{(1-\lambda \zeta^2)^{1/2}}d|\zeta|
	\end{align*} 
	for some $\lambda\in\mathbb{C}$ with $|\lambda|=1$.
	
\end{cor}
\begin{proof}[\bf Proof of Theorem \ref{Th-2.2}]
	Let $f\in\mathcal{N}(\beta)$, and from \eqref{Eq-2.3}, we obtain $\phi(0)=0$. Then  by using the Schwarz lemma, we have
	\begin{align*}
		\bigg|\frac{\frac{f^{\prime\prime}(z)}{f^{\prime}(z)}}{\frac{zf^{\prime\prime}(z)}{f^{\prime}(z)}-2(\beta-1)}\bigg|^2\leq|z|^2
	\end{align*}
	which implies that
	\begin{align*}
		\bigg|\frac{f^{\prime\prime}(z)}{f^{\prime}(z)}\bigg|^2\leq4\left(\beta-1\right)^2|z|^2-4(\beta-1)|z|^2{\rm Re}\left(\frac{zf^{\prime\prime}(z)}{f^{\prime}(z)}\right)+|z|^4\bigg|\frac{zf^{\prime\prime}(z)}{f^{\prime}(z)}\bigg|^2.
	\end{align*}
	Thus, we have
	\begin{align}
		(1-|z|^4)\bigg|\frac{f^{\prime\prime}(z)}{f^{\prime}(z)}\bigg|^2\leq4\left(\beta-1\right)^2|z|^2-4(\beta-1)|z|^2{\rm Re}\left(\frac{zf^{\prime\prime}(z)}{f^{\prime}(z)}\right).
	\end{align}
	Multiplying both sides the above inequality by $(1-|z|^4)$, we obtain 
	\begin{align*}
		(1-|z|^4)^2&\bigg|\frac{f^{\prime\prime}(z)}{f^{\prime}(z)}\bigg|^2+4\left(\beta-1\right)|z|^2(1-|z|^4){\rm Re}\left(\frac{zf^{\prime\prime}(z)}{f^{\prime}(z)}\right)\\&\leq4\left(\beta-1\right)^2|z|^2(1-|z|^4).
	\end{align*}
	Adding  $\left(2(\beta-1)|z|^2\bar{|z|}\right)^2$ both side of the above inequality, leads to
	\begin{align*}
		(1-|z|^4)^2&\bigg|\frac{f^{\prime\prime}(z)}{f^{\prime}(z)}\bigg|^2+4\left(\beta-1\right)|z|^2(1-|z|^4){\rm Re}\left(\frac{zf^{\prime\prime}(z)}{f^{\prime}(z)}\right)+4\left(\beta-1\right)^2|z|^4|\bar{z}|^2\\&\leq4\left(\beta-1\right)^2|z|^2(1-|z|^4)+4\left(\beta-1\right)^2|z|^4|\bar{z}|^2\nonumber.
	\end{align*}
	Multiplying both side by $|z|$, then by simple calculation, we obtain
	\begin{align*}
		\bigg|(1-|z|^4)\frac{zf^{\prime\prime}(z)}{f^{\prime}(z)}+2\left(\beta-1\right)|z|^4\bigg|\leq2\left(\beta-1\right)|z|^2
	\end{align*}
	which implies that
	\begin{align*}
		\frac{-2\left(\beta-1\right)|z|^2}{1+|z|^2}\leq{\rm Re}\left(\frac{zf^{\prime\prime}(z)}{f^{\prime}(z)}\right)\leq\frac{2\left(\beta-1\right)|z|^2}{1-|z|^2}.
	\end{align*}
	Let $z=re^{i\theta}$. It is easy to see that
	\begin{align*}
		\frac{-2\left(\beta-1\right) r}{1+r^2}\leq\dfrac{\partial}{\partial r}\left(\log|f^{\prime}(re^{i\theta})|\right)\leq\frac{2\left(\beta-1\right) r}{1-r^2}.
	\end{align*}
	Integrating the above estimate with respect to $r$, we obtain
	\begin{align*}
		\frac{1}{(1+|z|^2)^{\beta-1}}\leq|f^{\prime}(z)|\leq\frac{1}{(1-|z|^2)^{\beta-1}}.
	\end{align*}
	Next, the growth part of the theorem follows from the upper bound
	\begin{align*}
		|f^{\prime}(re^{i\theta})|=\bigg|\int_{0}^{r}f^{\prime}(re^{i\theta})e^{i\theta} dt\bigg|\leq\int_{0}^{r}|f^{\prime}(re^{i\theta})| dt\leq\int_{0}^{r}\frac{1}{(1-t^2)^{\beta-1}} dt
	\end{align*}
	which implies that
	\begin{align*}
		|f(z)|\leq\int_{0}^{|z|}\frac{1}{(1-\xi^2)^{\beta-1}} d|\xi|
	\end{align*}
	for all $z\in\mathbb{D}$.\vspace{2mm} 
	
	It is well-known that if $f(z_0)$ is a point of minimum modulus on the image of the circle $|z|=r$ and $\gamma=f^{-1}(\Gamma)$, where $\Gamma$ is the line segment from $0$ to $f(z_0)$, then 
	\begin{align*}
		|f(z)|\geq	|f(z_0)|=\int_{\Gamma}|dw|=\int_{\gamma}|f^{\prime}(\zeta)||d\zeta|\geq\int_{0}^{r}\frac{1}{(1+\xi^2)^{\beta-1}} d|\xi|.
	\end{align*}
	Thus all the desired inequalities are established. For the function \( f_{\beta, \lambda}(z) \) in the main result, the sharpness part can be shown easily, hence we omit the details.
\end{proof}
We will find the sharp bound of the pre-Schwarzian and Schwarzian norms in terms of value $f^{\prime\prime}(0)$ in the class $\mathcal{N}(\beta)$, under the assumption  that $f^{\prime\prime}(0)=0$. The following lemma will paly a key role to prove the result.
\begin{lemA}{\rm \cite{Carrasco-Hernández-AMP-2023}}
	If $\phi(z):\mathbb{D}\rightarrow\mathbb{D}$ be analytic function, then 
	\begin{align*}
		\frac{|\phi(z)|^2}{1-|\phi(z)|^2}\leq\frac{(\phi(0)+|z|)^2}{(1-|\phi(0)|)^2(1-|z|^2)|)}
	\end{align*}
\end{lemA}
We establish a sharp bound on the pre-Schwarzian norm for  $f\in\mathcal{N}(\beta)$.
\begin{thm}\label{Th-2.3}
	For $\beta>1$, let $f\in\mathcal{N}(\beta)$ and for all $z\in\mathbb{D}$, then the pre-Schwarzian norm estimate
	\begin{align*}
		||P_f||\leq2\left(\beta-1\right)  .
	\end{align*}
	The inequality is sharp.
\end{thm}
The following result for the class \( \mathcal{N}(3/2) \) is a direct consequence of Theorem \ref{Th-2.3}.
\begin{cor}\label{Cor-2.1}
	For $\beta=3/2$, let $f\in\mathcal{N}(3/2)$ and for all $z\in\mathbb{D}$, then the pre-Schwarzian norm estimate
	\begin{align*}
		||P_f||\leq1 .
	\end{align*}
	The inequality is sharp.
\end{cor}
\begin{proof}[\bf Proof of Theorem \ref{Th-2.3}]
	Since $\phi(z)=z\xi(z)$, with $|\xi(z)|<1$, then in \eqref{Eq-2.3} we obtain
	\begin{align*}
		\sup_{z\in\mathbb{D}}(1-|z|^2)\bigg|\frac{f^{\prime\prime}(z)}{f^{\prime}(z)}\bigg|&\leq\sup_{z\in\mathbb{D}}(1-|z|^2)\frac{2\left(\beta-1\right)|z\xi(z)|}{1-|z|^2|\xi(z)|}\\&\leq2\left(\beta-1\right) \sup_{0\leq r\leq 1}\frac{r(1-r^2)}{(1-r^2)}\\&=2\left(\beta-1\right).
	\end{align*}
	Thus, we see that $||P_f||\leq2\left(\beta-1\right)$. 
	To show that the inequality is sharp, we consider the extremal function is given by
	\begin{align*}
		f_{\beta}(z)=	\int_{0}^{z}\frac{1}{(1-\xi^2)^{\left(\beta-1\right)}} d\xi.
	\end{align*}
	It can be easily shown that $||P_{f_{\beta}}||=2\left(\beta-1\right).$ This completes the proof.
\end{proof}
\subsection{\bf Sharpness of Corollary \ref{Cor-2.1}}
For $\beta=3/2$, it follows from  that
\begin{align*}
	\frac{f_{3/2}^{\prime\prime}}{f_{3/2}^{\prime}}(z)=-\frac{1}{(1-z^2)}\;\;\mbox{and}\;\;|P_{f_{3/2}}|=\frac{1}{1-|z|^2}.
\end{align*}
A simple computation thus yields that
\begin{align*}
	||P_{f_{3/2}}||=\sup_{z\in\mathbb{D}}\left(1-z^2\right)|P_{f_{3/2}}|=\sup_{z\in\mathbb{D}}\left(1-|z|^2\right)\frac{1}{1-|z|^2}=1
\end{align*}
and we see the constant $1$ is best possible.\vspace{2mm}

\begin{thm}\label{Th-2.4}
	For $\beta>1$, let $f\in\mathcal{N}(\beta)$,   for all $z\in\mathbb{D}$, then the Schwarzian norm estimate
	\begin{align*}
		||S_f||=(1-|z|^2)^2|S_f(z)|\leq2\left(\beta^2-1\right).
	\end{align*}
	The inequality is sharp.
\end{thm}
The following result for the class $\mathcal{N}(3/2)$ follows directly from Theorem \ref{Th-2.4}.
\begin{cor}
	For $\beta=3/2$, let $f\in \mathcal{N}(3/2)$ be of the form \eqref{Eq-1.3}, then the Schwarzian norm estimate 
	\begin{align*}
		||S_f||=(1-|z|^2)^2|S_f(z)|\leq\frac{5}{2}.
	\end{align*}
	The estimate is sharp.
\end{cor}
	\begin{proof}[\bf Proof of Theorem \ref{Th-2.4}]
		From \eqref{Eq-2.3}, we have
		\begin{align*}
			\left(\frac{f^{\prime\prime}(z)}{f^{\prime}(z)}\right)=-\frac{2(\beta-1)\phi(z)}{(1-z\phi(z))}.
		\end{align*}
		A simple calculation gives that
		\begin{align*}
			S_f(z)=-2\left(\beta-1\right)\left[\frac{\phi^{\prime}(z)+\beta\phi^2(z)}{(1-z\phi(z))^2}\right].
		\end{align*}
		By using triangle inequality and Schwarz pick lemma, we obtain
		\begin{align}\label{Eq-2.6}
			(1-|z|^2)^2|S_f|&\leq2\left(\beta-1\right)\bigg|\phi^{\prime}(z)+\beta\phi^2(z)\bigg|\frac{(1-|z|^2)^2}{|1-z\phi(z)|^2}\\&\nonumber=\frac{2\left(\beta-1\right)(1-|z|^2)^2}{|1-z\phi(z)|^2}\left(\frac{1-|\phi(z)|^2}{1-|z|^2}+\beta|\phi(z)|^2\right).
		\end{align}
		We define the function $\Psi(z):\mathbb{D}\rightarrow\mathbb{D}$ such that
		\begin{align*}
			\Psi(z):=\frac{\bar{z}-\phi(z)}{1-z\phi(z)}.
		\end{align*}
		Since $\phi(\mathbb{D})\subseteq\mathbb{D}$ then $(1-|z|^2)(1-|z\phi(z)|^2)>0$, it follows that 
		\begin{align*}
			|\bar{z}-\phi(z)|^2<|1-z\phi(z)|^2.
		\end{align*}
		Thus, we can conclude  that $|\Psi(z)|^2<1$. It is easy to see that
		\begin{align*}
			1-|\Psi(z)|^2=\frac{(1-|\phi(z)|^2)(1-|z|^2)}{|1-z\phi(z)|^2}
		\end{align*}
		and
		\begin{align}\label{Eq-2.7}
			\frac{(1-|z|^2)^2}{|1-z\phi(z)|^2}=\frac{(1-|\Psi(z)|^2)(1-|z|^2)}{(1-|\phi(z)|^2)}.
		\end{align}
		If we replace the expression \eqref{Eq-2.7} in \eqref{Eq-2.6}, then we have
		\begin{align}\label{Eq-2.8}
			(1-|z|^2)^2|S_f(z)|\leq2\left(\beta-1\right)(1-|\Psi_1(z)|^2)\left(1+\beta\frac{|\phi(z)|^2(1-|z|^2)}{(1-|\phi(z)|^2)}\right).
		\end{align}
		Since $f^{\prime\prime}(0)=0$ implies that $\phi(0)=0$, using Lemma A, we obtain
		\begin{align}\label{Eq-2.9}
			\frac{|\phi(z)|^2}{1-|\phi(z)|^2}\leq\frac{|z|^2}{1-|z|^2}.
		\end{align}
		Using \eqref{Eq-2.9} in \eqref{Eq-2.8}, we obtain
		\begin{align*}
			(1-|z|^2)^2|S_f(z)|\leq2\left(\beta-1\right)(1-|\Psi(z)|^2)\left(1+\beta|z|^2\right).
		\end{align*}
		Again, since $1-|\Psi(z)|^2$$\leq$1, then
		\begin{align*}
			\sup_{z\in\mathbb{D}}(1-|z|^2)^2|Sh(z)|&\leq2\left(\beta-1\right)\left(1+\beta\right)\\&=2\left(\beta^2-1\right).\nonumber
		\end{align*}
		Thus the desired inequality is obtained. The sharpness follows from the example \ref{Example-2.1}.
		\end{proof}

\begin{exm}\label{Example-2.1}
	The family of parameterized functions defined as:
	\begin{align*}
		f_\beta (z)=\int_{0}^{z}\frac{1}{(1-\xi^2)^{\beta-1}} d\xi,\;\;\;\;\mbox{for}\;\;\beta>1
	\end{align*}
	maximizes the Schwarzian norm defined as:
	\begin{align*}
		||S_f||=\sup_{z\in\mathbb{D}}(1-|z|^2)^2|S_f| 
	\end{align*}
	and from this, the sharpness of the inequality holds for $\beta>0$. Note that
	\begin{align*}
		\frac{f^{\prime\prime}_\beta}{f_\beta^{\prime}}(z)=\frac{2(\beta-1)z}{1-z^2}\;\;\mbox{and}\;\;S_{f_\beta}=\frac{2(\beta-1)}{(1-z^2)^2}\left(1+\beta|z|^2\right)
	\end{align*}
	which calculates
	\begin{align*}
		||S_{f_\beta}||&=\sup_{z\in\mathbb{D}}(1-|z|^2)^2|S_{f_\beta}|\\&=2\left(\beta^2-1\right).
	\end{align*}
\end{exm}
For the value $|f^{\prime\prime}(0)|$ is not necessarily zero, we give a bound for the quantity $(1-|z|^2)^2|S_f(z)|$ with  $f\in\mathcal{N}(\beta)$ in the following.
\begin{thm}\label{Th-2.5}
	For $\beta>1$, let $f\in\mathcal{N}(\beta)$, for all $z\in\mathbb{D}$, then the inequality 
	\begin{align*}
		\gamma=|\phi(0)|=\frac{|f^{\prime\prime}(0)|}{2\left(\beta-1\right)},
	\end{align*}
	then 
	\begin{align*}
		(1-|z|^2)^2|S_f(z)|\leq2\left(\beta-1\right)\left(1+\beta\frac{1+\gamma}{1-\gamma}\right).
	\end{align*}
\end{thm}
We have the following immediate result from Theorem \ref{Th-2.5} for the class $\mathcal{N}(3/2)$.
\begin{cor}
		For $\beta=3/2$, let $f\in\mathcal{N}(3/2)$, for all $z\in\mathbb{D}$, then the inequality 
		\begin{align*}
			\gamma=|\phi(0)|=|f^{\prime\prime}(0)|,
		\end{align*}
		then 
		\begin{align*}
			(1-|z|^2)^2|S_f(z)|\leq\left(1+\frac{3}{2}\frac{1+\gamma}{1-\gamma}\right).
		\end{align*}
\end{cor}
\begin{proof}[\bf Proof of Theorem \ref{Th-2.5}]
	Let $\gamma=|\phi(0)|$. Applying the Lemma A, we  calculate
	\begin{align*}
		\frac{|\phi(z)|^2}{1-|\phi(z)|^2}\leq\frac{(\gamma+|z|)^2}{(1-\gamma^2)(1-|z|^2)}.
	\end{align*}
	If we substitute above inequality in \eqref{Eq-2.6}, we get
	\begin{align*}
		(1-|z|^2)^2|S_f(z)|\leq2\left(\beta-1\right)(1-|\Phi_1(z)|^2)\left(1+\beta\frac{(\gamma+|z|)^2}{(1-\gamma^2)}\right).
	\end{align*}
	From the fact that $|z|<1$ and  $1-|\Phi_1(z)|^2\leq1$, we can easily calculate 
	\begin{align*}
		(1-|z|^2)^2|S_f(z)|\leq2\left(\beta-1\right)\left(1+\beta\frac{1+\gamma}{1-\gamma}\right).
	\end{align*}
	This is the desired bound.
\end{proof}
\section{\bf {Schwarzian derivatives for Harmonic mapping with fixed analytic part}}\label{Sec-3}
In this article, we are concerned with a class of functions $\mathcal{N}(\beta)$, $1<\beta\leq3/2$, defined by
\begin{align*}
	\mathcal{N}(\beta)=\bigg\{f\in\mathcal{A}:\;{\rm Re}\;\bigg(1+\frac{zf^{\prime\prime}(z)}{f^{\prime}(z)}\bigg)<\beta\;\;\mbox{for}\;z\in\mathbb{D}\bigg\}.
\end{align*}
A complex-valued function $f=u+iv$ is harmonic if $u$ and $v$ are real-harmonic in $\mathbb{D}$. Every harmonic function $f$ has the canonical representation $f=h+ \overline{g}$, where $h$ and $g$ are analytic in $\mathbb{D}$ known respectively as the analytic and co-analytic parts of $ f $. A locally univalent harmonic function $f$ is said to be sense-preserving if the Jacobian of $f$, defined by $J_{f}(z):=|h'(z)|^{2}-|g'(z)|^{2}$, is positive in $\mathbb{D}$ and sense-reversing if $J_{f}(z)$ is negative in $\mathbb{D}$. Let $\mathcal{H}$ be the class of all complex-valued harmonic functions $f=h+\overline{g}$ defined in $\mathbb{D}$, where $h$ and $g$ are analytic 
in $\mathbb{D}$ such that $h(0)=h'(0)-1=0$ and $g(0)=0$. For more details, we refer to \cite{Duren-Har-2004}.\vspace{1.2mm}

Since harmonic mappings are natural extension of analytic maps on unit disk $\mathbb{D}$, in this section, we introduce a new subclass $\mathcal{N}_{\mathcal{H}}(\beta)$ for $\beta>0$ of harmonic mappings which is defined by 
\begin{align*}
	\mathcal{N}_{\mathcal{H}}(\beta):=\{f=h+\bar{g} \in\mathcal{H} : h\in \mathcal{N}(\beta)\; \mbox{and}\;g^{\prime}=\omega h^{\prime}\},
\end{align*}
where $\omega : \mathbb{D}\to\mathbb{D}$ is an analytic function.\vspace{2mm}

It is easy to see that $\mathcal{N}(\beta)\subseteq \mathcal{N}_{\mathcal{H}}(\beta)$ in $ \mathbb{D}$ in the sense that harmonic mappings of the form $f=h+\bar{g}$ \cite{Lewy-BAMS-1936} is a generalization of analytic functions $h$. The Schwarzian norm estimate is understood from the article \cite{Ahamed-Allu-Hossain-MM-2025,Chuaqui-Duren-Osgood-JAM-2003} but the Schwarzian norm or pre-Schwarzian norm estimate is not explored yet for a class of harmonic mappings associated with the class $\mathcal{N}(\beta)$. \vspace{2mm}

Inspired by this fact, our main aim is to derive estimates for the modulus of the pre-Schwarzian norm for the functions in $\mathcal{N}_{\mathcal{H}}(\beta)$. In fact, we obtain the following result which will provide a sharp estimate of the pre-Schwarzian norm for functions in the class $\mathcal{N}_{\mathcal{H}}(\beta)$.
\begin{thm}\label{Th-3.1}
	If $f=h+\bar{g}\in\mathcal{N}_{\mathcal{H}}(\beta)$ for $\beta>0$ with $g^{\prime}=\omega h^{\prime}$, where $\omega : \mathbb{D}\to\mathbb{D}$ is an analytic function, then $||P_{f}||\leq 4\beta-3$. The estimate $4\beta-3$ is best possible.
\end{thm}
\begin{proof}[\bf Proof of Theorem \ref{Th-3.1}]
	Let $f=h+\bar{g}\in\mathcal{N}_{\mathcal{H}}(\beta)$. Then $h$ holds the subordination relation 
	\begin{align*}
		\frac{zh^{\prime\prime}(z)}{h^{\prime}(z)}\prec\frac{-2(\beta-1)z}{1-z},
	\end{align*}
	which gives us
	\begin{align}\label{Eq-55.22}
		\bigg|\frac{h^{\prime\prime}(z)}{h^{\prime}(z)}\bigg|\leq\frac{2(\beta-1)}{1-|z|}.
	\end{align}
	By the Schwarz-Pick lemmma and \eqref{Eq-55.22}, we have 
	\begin{align*}
		||P_{f}||&=\sup_{z\in\mathbb{D}}\left(1-|z|^2\right)|P_{f}|\\&=\sup_{z\in\mathbb{D}}\left(1-|z|^2\right)\bigg|\frac{h^{\prime\prime}(z)}{h^{\prime}(z)}-\frac{\bar{\omega}(z)\omega^{\prime}(z)}{1-|\omega(z)|^2}\bigg|\\&\leq\sup_{z\in\mathbb{D}}\left(1-|z|^2\right)\left(\bigg|\frac{h^{\prime\prime}(z)}{h^{\prime}(z)}\bigg|+\bigg|\frac{\bar{\omega}(z)\omega^{\prime}(z)}{1-|\omega(z)|^2}\bigg|\right)\\&\leq\sup_{z\in\mathbb{D}}\left(1-|z|^2\right)\left(\frac{2(\beta-1)}{1-|z|}+\frac{|\bar{\omega}(z)|}{1-|z|^2}\right)\\&\leq\sup_{z\in\mathbb{D}}\left(2(\beta-1)(1+|z|)+|\omega(z)|\right)\\&=4(\beta-1)+1\\&=4\beta-3.
	\end{align*}
	We now show that the estimate is best possible. For
	\begin{align*}
		\begin{cases}
			\sqrt{1-\beta}\leq t<1,\;\;\;\mbox{when}\;\;\beta\in[0,1)\vspace{2mm}\\\dfrac{\beta-1}{\beta+1}\leq t<1,\;\;\;\;\;\;\mbox{when}\;\;\beta\in[1,\infty),
		\end{cases}
	\end{align*}
	we consider the function $f_t=h_t+\bar{g_t}\in\mathcal{H}$ with the second complex dilatation $\omega_t(z)=\frac{z-t}{1-tz}$ and $h_t(z)$ be such that
	\begin{align*}
		\frac{zh_t^{\prime\prime}(z)}{h_t^{\prime}(z)}=\frac{-2(\beta-1)z}{1-z}.
	\end{align*}
	Then clearly $f_t\in\mathcal{G}_{\mathcal{H}}(\beta)$ for all t in $[0,\infty)$. A simple computations yields that
	\begin{align*}
		\frac{\overline{\omega_t}(z)\omega^{\prime}_t(z)}{1-|\omega_t(z)|^2}=\frac{\overline{z}-t}{(1-tz)(1-|z|^2)}.
	\end{align*}
	Consequently, we have
	\begin{align*}
		||P_{f_t}||&=\sup_{z\in\mathbb{D}}(1-|z|^2)\bigg|\frac{h^{\prime\prime}_t(z)}{h^{\prime}_t(z)}-\frac{\overline{\omega}(z)\omega^{\prime}(z)}{1-|\omega(z)|^2}\bigg|\\&=\sup_{z\in\mathbb{D}}(1-|z|^2)\bigg|\frac{2(\beta-1)}{1-|z|}-\frac{\bar{z}-t}{(1-tz)(1-|z|^2)}\bigg|.
	\end{align*}
	We set
	\begin{align*}
		M_t=\sup_{z\in\mathbb{D}}(1-|z|^2)\bigg|\frac{2(\beta-1)}{1-|z|}-\frac{\bar{z}-t}{(1-tz)(1-|z|^2)}\bigg|
	\end{align*}
	\begin{align}\label{Eq-5.4}
		\geq \sup_{z\in[0,1)}\bigg|2(\beta-1)(1+r)-\frac{r-t}{1-tr}\bigg|=\sup_{z\in[0,1)}\;\phi(r),
	\end{align}
	where 
	\begin{align*}
		\phi(r)=2(\beta-1)(1+r)-\frac{r-t}{1-tr}.
	\end{align*}
	A simple computation shows that
	\begin{align*}
		\phi^{\prime}(r)=2(b-1)-\frac{t(r-t)}{(1-tr)^2} - \frac{1}{(1-tr)}
	\end{align*}
	and 
	\begin{align*}
		\phi^{\prime\prime}(r)=-\frac{2t(1-t^2)(1-rt)}{(1-rt)^4}<0\;\;\mbox{for all}\;r\in[0,1).
	\end{align*}
	Now, $\phi^{\prime}(r)=0$ gives
	\begin{align*}
		r=r_0:=\frac{2\beta-2-\sqrt{-\beta t^2 +\beta+t^2-1}}{2\beta t-2t}.
	\end{align*}
	Thus the maximum value of $\phi$ is attained at $r_0\in [0, 1)$. Hence, we have
	\begin{align}\label{Eq-5.5}
		M_t=\phi(r_0)=\frac{2\beta t+2\beta-2t-2\sqrt{2}\sqrt{-\beta t^2+\beta+t^2-1}-1}{t}.
	\end{align}
	In view of \eqref{Eq-5.4} and \eqref{Eq-5.5}, it is clear that $M_t\leq||P_{f_t}||\leq 4\beta-3$. We note that $M_t$ is increasing function for t and $M_t\rightarrow4\beta-3$ as $t\rightarrow1$. This shows that estimate is the best possible.
\end{proof}
\section{\bf Estimate of Bloch constant for the class $\mathcal{N}_{\mathcal{H}}(\beta)$}\label{Sec-4}
Let $\mathcal{H}ar(\mathbb{D})$ denote the family of continuous complex-valued functions which are harmonic in the open unit disk $\mathbb{D} = \{z \in \mathbb{C} : |z| < 1\}$, and let $\mathcal{H}ol(\mathbb{D})$ denote the class of holomorphic functions $f$ in $\mathbb{D}$ withthe normalization $f(0) = f'(0) - 1 = 0$ in $\mathbb{D}$. We note that $\mathcal{H}ol(\mathbb{D}) \subset \mathcal{H}ar(\mathbb{D})$. Further, let $\mathcal{S} := \mathcal{S}_{\mathcal{H}ol}$ be the subclass of $\mathcal{H}ol(\mathbb{D})$ which are additionally univalent in $\mathbb{D}$. Clunie and Sheil-Small in \cite{Clunie-Sheil-Small-AASFMSAIM-1984} developed the fundamental theory of functions $f \in \mathcal{H}ar(\mathbb{D})$ with the normalization $f(0) = h'(0) - 1 = 0$ in $\mathbb{D}$. Following Clunie and Sheil-Small's notation, we next denote by $\mathcal{S}_{\mathcal{H}ar}$, the subclass of $\mathcal{H}ar(\mathbb{D})$ consisting of univalent and sense-preserving harmonic mappings $f = h +\overline{g}$ in $\mathbb{D}$, where $h$ and $g$ are normalized such that
\begin{align}\label{Eq-7.1}
	h(z) = z + \sum_{n=2}^{\infty} a_n z^n\;\mbox{and}\;g(z) = \sum_{n=1}^{\infty} b_n z^n.
\end{align}
Here, $h$ and $g$ are called the analytic part and the co-analytic part of $f$, respectively.

The analytic parts of harmonic mappings play a vital role in shaping their geometric properties. For instance, if $f=h+\overline{g}$ is a sense-preserving harmonic mapping and $h$ is convex univalent, then $f \in \mathcal{S}_{\mathcal{H}ar}$ and maps $\mathbb{D}$ onto a close-to-convex domain \cite{Clunie-Sheil-Small-AASFMSAIM-1984}. In \cite{Kanas-Klimek-Smet-BKMS-2014,Kanas-Klimek-Smet-BM-2016}, a class of functions $f=h+\overline{g}\in \mathcal{S}_{\mathcal{H}ar}$ has been studied, where $h$ and $g$ are given by (1.2), such that $b_1=2\beta/3\in (0, 1)$, $h$ is convex in $\mathbb{D}$ (or $h$ is a function with bounded boundary rotation) and the dilatation $\omega$ is of the form
\begin{align*}
	w(z) =(z+2\beta/3)/(1 + 2z\beta/3)
\end{align*}

For $\beta$ ($0 <2\beta/3<1$), let $\mathcal{F}_{\mathcal{H}}(\beta)$ denote the set of all harmonic functions $f=h+\overline{g} \in \text{Har}(\mathbb{D})$, with $g'(0) = b_1 =2\beta/3$ and
\begin{align}\label{Eq-7.2}
	g'(z)=\omega(z)h'(z)\;\mbox{and}\;\omega(z)\in\mathcal{N}(\beta), \;(z \in \mathbb{D})
\end{align}
where $\omega$ is the Möbius selfmap of $\mathbb{D}$ of the form 
\begin{align}\label{Eq-4.3A}
	\omega(z) = \frac{z+2\beta/3}{1+2z\beta/3}.
\end{align}
 The function $w$ has the series expansion
\begin{align}\label{Eq-7.3}
	\omega(z) = 2\beta/3 + a_1 z + a_2 z^2 + \cdots\;\; (z \in \mathbb{D}, a_i \in \mathbb{C}, i = 1, 2, \ldots.)
\end{align}
If $f \in\mathcal{N}_{\mathcal{H}}(\beta)$, then according to the form \eqref{Eq-4.3A}, and the relation $\omega=g'/h'$, we have $a_0=b_1=2\beta/3$,
\begin{align}\label{Eq-7.4}
	\frac{|r -2\beta/3|}{1 -r2\beta/3} \le |\omega(z)| \leq \frac{r +2\beta/3}{1 + 2\beta r/3},
\end{align}
and
\begin{align}
	|a_n| \leq 1 - |a_0|^2 \;\;(n = 1, 2, \ldots),\;\;|\omega'(z)| \le \frac{1 - |\omega(z)|^2}{1 - |z|^2}\;\;(z \in \mathbb{D}).
\end{align}
The classical Bloch theorem asserts the existence of a positive constant $b$ such that for any holomorphic mapping $f$ of the unit disk $\mathbb{D}$, with the normalization $f'(0) = 1$, the image $f(\mathbb{D})$ contains a Schlicht disk of radius $b$. By Schlicht disk, we mean a disk which is the univalent image of some region in $\mathbb{D}$. The Bloch constant is defined as the \textquotedblleft best\textquotedblright{} such constant, that is supremum of such constants $b$. Chen et al.\cite{Chen-Gauthier-Hengartner-PAMS-2000} estimated Bloch constant for harmonic mappings.\vspace{1.2mm}

A function $f \in \mathcal{H}ar(\mathbb{D})$ is called a harmonic Bloch mapping if and only if
\begin{align}
	\mathcal{B}_f= \sup_{z,\omega \in \mathbb{D}, z \ne \omega} \frac{|f(z) - f(\omega)|}{\mathcal{Q}(z,\omega)} < \infty,
\end{align}
where
\begin{align*}
	\mathcal{Q}(z,\omega)=\frac{1}{2} \log \left( \frac{1 + \left| \frac{z-\omega}{1-\overline{z}\omega} \right|}{1 - \left| \frac{z-\omega}{1-\overline{z}\omega} \right|} \right) = \text{arctanh} \left| \frac{z - \omega}{1 -\overline{z}\omega} \right|
\end{align*}
denotes the hyperbolic distance between $z$ and $\omega$ in $\mathbb{D}$, and $\mathcal{B}_f$ called the Bloch’s constantof $f$. In \cite{Colonna-IUMJ-1989} Colonna proved that
\begin{align}\label{Eq-7.7}
	\mathcal{B}_f = \sup_{z \in \mathbb{D}} (1 - |z|^2) \Lambda f
\end{align}
where
\begin{align*}
	\Lambda_f &= \Lambda_f (z) = \max_{0 \leq \theta \le 2\pi} \left| f_z(z) - e^{-2i\theta} {f_{\bar{z}}(z)} \right| = |f_z(z)| + |f_{\bar{z}}(z)|\\&=|h'(z)| + |g'(z)| = |h'(z)|(1 + |\omega(z)|).
\end{align*}
Moreover, the set of all harmonic Bloch mappings forms a complex Banach space with the norm $\|\cdot\|$ given by
\begin{align*}
	||f|| = |f(0)| + \sup_{z\in\mathbb{D}} (1 - |z|^2)\Lambda_f (z).
\end{align*}
This definition agrees with the notion of the Bloch’s constant for analytic functions. Recently, many authors have studied Bloch’s constant for harmonic mappings. (see \cite{Chen-Gauthier-Hengartner-PAMS-2000,Kanas-Klimek-Smet-BM-2016,Liu-SCS-2009}).\vspace{1.2mm}

Inspired by the work in \cite[Proposition 1]{Maharana-Prajapat-Srivastava-PNA-2017} and \cite{Kanas-Maharana-Prajapat-JMAA-2019} for a different class of functions, we prove a similar result for our setting. We then use this result to prove our main theorem.
\begin{lem}\label{Lem-7.1}
	If $f\in\mathcal{N}(\beta)$ of the form $f(z)=z+\sum_{n=2}^{\infty}a_nz^n$, then for $|z|=r<1$, the following statements are
	\begin{enumerate}
		\item[(a)] $\bigg|\dfrac{zf^{\prime\prime}(z)}{f^{\prime}(z)}\bigg|\leq\dfrac{2(\beta-1) r}{1-r}$. The inequality is sharp. The equality is attended for the function
		\begin{align}\label{Eq-7.8}
			g_\beta(z)=\int(1-z)^{2(\beta-1)}dz. 
		\end{align}\vspace{2mm}
		
		\item[(b)] 	$(1-|z|)^{2(\beta-1)}\leq |f^{\prime}(z)|\leq (1+|z|)^{2(\beta-1)}$. The inequality is sharp, with equality achieved by the function given by \eqref{Eq-7.8}.
	\end{enumerate}
\end{lem}
\begin{proof}[\bf Proof of Lemma \ref{Lem-7.1}]
	From the definition of $f\in\mathcal{N}(\beta)$, we have
	\begin{align*}
		{\rm Re  }\left(1+\frac{zf^{\prime\prime}(z)}{f^{\prime}(z)}\right)\prec1-\frac{2(\beta-1) z}{1-z}
	\end{align*}
	which implies that
	\begin{align*}
		{\rm Re  }\left(\frac{zf^{\prime\prime}(z)}{f^{\prime}(z)}\right)\prec-\frac{2(\beta-1) z}{1-z}=h(z).
	\end{align*}
	The result $(a)$ follows easily. To show the sharpness note that 
	\begin{align*}
		\frac{zg^{\prime\prime}(z)}{g^{\prime}(z)}=-\frac{2(\beta-1) z}{1-z}.
	\end{align*}
	This equality, when $z=r, 0\leq r<1.$ To prove the result $(b)$ from the definition of the function $f\in\mathcal{N}(\beta)$ and a well known subordination result in \cite{Suffridge-DMJ-1970}
	\begin{align*}
		f^{\prime}(z)&\prec\exp\left(\int_{0}^{z}\frac{h(t)}{t}dt\right)\\&=\exp\left(\int_{0}^{z}\frac{-2(\beta-1)}{1-t}dt\right)=(1-t)^{2(\beta-1)}.
	\end{align*}
	One can easily show the sharpness in $(b)$.
\end{proof}
This section focuses on finding bounds for the Bloch constant within the co-analytic part of functions in the class $\mathcal{N}_{\mathcal{H}}(\beta)$.
\begin{thm}\label{Th-7.1}
	Let $\beta\in(1,3/2]$ and let$f\in\mathcal{N}_{\mathcal{H}}(\beta)$. The bloch constant $\mathcal{B}_f$ is bounded by 
	\begin{align*}
		\mathcal{B}_f\leq\frac{(3+2\beta)(1-r_0^2)(1+r_0)^{2\beta-1}}{3+2\beta r_0}
	\end{align*}
	where $r_0$ is the unique root of the equation $(4\beta-3) + (4\beta^2 - 4\beta - 6)r + (-8\beta - 3)r^2 - 4\beta^2 r^3=0$ in the interval $(0,1)$.
\end{thm}
\begin{proof}[\bf Proof of Theorem \ref{Th-7.1}]
	Let $f=h+\overline{g}\in\mathcal{N}_{\mathcal{H}}(\beta)$ and  $h\in\mathcal{N}(\beta)$. Using Lemma \ref{Lem-7.1} along with \eqref{Eq-7.7} and \eqref{Eq-7.4}, we obtain
	\begin{align*}
		\mathcal{B}_f&=\sup_{z\in\mathbb{D}}(1-|z|^2)|h^{\prime}(z)|(1+|\omega(z)|)\\&\leq \sup_{z\in\mathbb{D}}(1-r^2)(1+r)^{2(\beta-1)}\left(1+\frac{r+2\beta/3}{1+2\beta r/3}\right)=(3+2\beta)\sup_{0\leq r<1}\xi(r),
	\end{align*}
	where
	\begin{align*}
		\xi(r):=\frac{(1-r^2)(1+r)^{2\beta-1}}{3+2\beta r}.
	\end{align*}
	The derivative of $\xi(r)$ is equal to zero if $\eta(r) = 0$ for $r \in (0, 1)$, where
	\begin{align*}
		\eta(r)=(4\beta-3) + (4\beta^2 - 4\beta - 6)r + (-8\beta - 3)r^2 - 4\beta^2 r^3.
	\end{align*}
	We note that $\eta(0) = (4\beta-3)> 0$ holds for the value $\beta>3/4$, and $\eta(1) = -8\beta - 12 < 0$, so that there exists a root $r_0 \in (0, 1)$ such that $\eta(r_0) = 0$. Now it suffices to prove that $r_0$ is unique. It is enough to prove that the derivative $\eta'(r) < 0$ for $r \in (0, 1)$ and $\alpha \in (0, 1)$. This holds by virtue of the inequality $\eta'(r) = (4\beta^2 - 4\beta - 6) + (-16\beta - 6)r - 12\beta^2 r^2<0$ always. Hence
	\begin{align*}
		\sup_{0\leq r<1}\xi(r)=\frac{(1-r_0^2)(1+r_0)^{2\beta-1}}{3+2\beta r_0},
	\end{align*}
	where $r_0$ is unique root of $\eta(r) = 0$ for $r \in (0, 1)$. This proves the result.
\end{proof}
\section{\bf Radius problems for the class $\mathcal{N}(\beta)$}\label{Sec-5}
Determining the radius of convexity and the radius of concavity for a given class of functions, and demonstrating the sharpness of these radii, is an important aspect of Geometric Function Theory.\vspace{1.2mm}

It is natural to study the following problems.
\begin{prob}\label{Pro-1}
	Can we determine the radius of convexity for the class $\mathcal{N}(\beta)$?
\end{prob}
\begin{prob}\label{Pro-2}
	Can we determine the radius of concavity for the class $\mathcal{N}(\beta)$?
\end{prob}
We investigate the radius of concavity and convexity for a certain class of functions, providing affirmative answers to Problems \ref{Pro-1} and \ref{Pro-2}.
In this section, we find a lower bound of the radius of concavity $ R_{\rm Co(p)} $ of  the class $ \mathcal{S}(p) $. Now, we consider functions $f$ in $ \mathcal{A} $ that map $ \mathbb{D} $ conformally onto a domain whose complement with respect to $ \mathbb{C} $ is convex and that satisfy the normalization $ f(1)=\infty $. We will denote these families of functions by $ {\rm Co}(A) $. Now $ f\in {\rm Co}(A) $ if, and only if, $ {T_f(z)}>0 $ for every $ z\in\mathbb{D} $, where $ f(0)=f^{\prime}(0)-1 $ and 
\begin{align}\label{Eq-4.1A}
	T_f(z)=\frac{2}{A-1}\left(\frac{(A+1)}{2}\left(\frac{1+z}{1-z}\right)-1-z\frac{f^{\prime\prime}(z)}{f^{\prime}(z)}\right),
\end{align}
where $A\in (1, 2]$.\vspace{2mm}

Inspired by \cite[Definition 1.1.]{Bhowmik-CMFT-2024} and \cite{Ahamed-Hossain-UJ-2026}, for an arbitrary class $\mathcal{F}$ of functions, we define the radius of concavity.
\begin{defi}
	The radius of concavity (w.r.t $ \mathcal{F} $), a subclass of $ \mathcal{A} $ is the largest number $ \mathrm{R}_{\mathcal{F}}\in (0,1] $  such that for each function $ f\in\mathcal{F} $, $ {\rm Re} \left(T_f(z)\right)>0 $ for all $ |z|< \mathrm{R}_{\mathcal{F}}$, where $ T_{f}(z) $ is defined in \eqref{Eq-1.3}.
\end{defi}
\begin{thm}\label{Th-5.1}
	Let $A \in (1,2]$ and $\beta <{(A+1)}/{4}$. If $f \in \mathcal{N}(\beta)$, then $\mathrm{Re}\,(T_f(z)) > 0$ for $|z| < R_{\beta,\mathrm{Co}(A)}$, where $R_{\beta,\mathrm{Co}(A)}$ is the unique root of the polynomial $\Phi_1(r) = 0$ in the interval $(0,1)$, given by
	\begin{equation*}
		\Phi_1(r) = (A + 4\beta - 3)r^2 - 2(A + 2\beta - 1)r + A + 1 - 4\beta.
	\end{equation*}
	The radius $R_{\beta,\mathrm{Co}(A)}$ is best possible. The extremal function is 
	\begin{align*}
		f_\beta(z) = \int_0^z (1-t)^{2(\beta-1)}\,dt=z + \sum_{k=2}^{\infty} \frac{(2 - 2\beta)_{k-1}}{k!} z^k.
	\end{align*}
\end{thm}
\begin{rem}
	The power series expansion of
	\begin{align*}
		f_\beta(z) = \int_0^z (1-t)^{2(\beta-1)}\,dt,
	\end{align*}we first expand the integrand using the general binomial series for $\vert{}t\vert{} < 1$ as 
	\begin{align*}
		(1-t)^{\alpha} = \sum_{n=0}^{\infty} \binom{\alpha}{n} (-1)^n t^n = \sum_{n=0}^{\infty} \frac{(-1)^n \alpha(\alpha-1)\cdots(\alpha-n+1)}{n!} t^n.
	\end{align*}Setting $\alpha = 2(\beta-1)$, the binomial coefficient simplifies via the Pochhammer symbol (or rising factorial) $(a)_n$, 
	\begin{align*}
		\binom{2(\beta-1)}{n} (-1)^n = \frac{(-1)^n}{n!} \prod_{k=0}^{n-1} (2\beta - 2 - k) = \frac{(2 - 2\beta)_n}{n!}.
	\end{align*}
	In terms of the Gauss hypergeometric function ${_2F_1}$, the function is expressed directly as 
	\begin{align*}
		f_\beta(z) = z \, {_2F_1}\left(1, \, 2 - 2\beta; \, 2; \, z\right).
	\end{align*}
\end{rem}
\begin{proof}[\bf Proof of Theorem \ref{Th-5.1}]
	By Theorem \ref{Th-2.1}(iii), any $f \in \mathcal{N}(\beta)$ satisfies
	\begin{equation}\label{eq-pre-sch-disk}
		\left| (1 - |z|^2) \frac{f''(z)}{f'(z)} + 2(\beta - 1)\bar{z} \right| \leq 2(\beta - 1).
	\end{equation}
	Dividing inequality \eqref{eq-pre-sch-disk} by $1 - |z|^2$ gives
	\begin{equation*}
		\left| \frac{f''(z)}{f'(z)} + \frac{2(\beta - 1)\bar{z}}{1 - |z|^2} \right| \leq \frac{2(\beta - 1)}{1 - |z|^2}.
	\end{equation*}
	Multiplying by $|z| = r$, we obtain
	\begin{equation*}
		\left| \frac{z f''(z)}{f'(z)} + \frac{2(\beta - 1)r^2}{1 - r^2} \right| \leq \frac{2(\beta - 1)r}{1 - r^2},
	\end{equation*}
	which leads directly to the sharp lower bound
	\begin{equation*}
		\mathrm{Re}\left( 1 + \frac{z f''(z)}{f'(z)} \right) \geq 1 - \frac{2(\beta - 1)r^2}{1 - r^2} - \frac{2(\beta - 1)r}{1 - r^2} = 1 - \frac{2(\beta - 1)r}{1 - r}.
	\end{equation*}
	Recall that the concavity operator $T_f(z)$ is defined by
	\begin{equation*}
		T_f(z) = \frac{2}{A-1} \left( \frac{A+1}{2} \frac{1+z}{1-z} - \left(1 + \frac{z f''(z)}{f'(z)}\right) \right).
	\end{equation*}
	Taking the real part for $z = r e^{i\theta}$ and applying 
	\begin{align*}
		\mathrm{Re}\left(\frac{1+z}{1-z}\right) \geq \frac{1-r}{1+r},
	\end{align*} we obtain
	\begin{align*}
		\mathrm{Re}\,(T_f(z)) &\geq \frac{2}{A-1} \left( \frac{A+1}{2} \left(\frac{1-r}{1+r}\right) - 1 + \frac{2(\beta - 1)r}{1 - r} \right) \\
		&= \frac{1}{(A-1)(1-r^2)} \left[ (A+1)(1-r)^2 - 2(1-r^2) + 4(\beta-1)r(1+r) \right] \\
		&= \frac{1}{(A-1)(1-r^2)} \left[ (A + 4\beta - 3)r^2 - 2(A + 2\beta - 1)r + A + 1 - 4\beta \right] \\
		&:= \frac{1}{(A-1)(1-r^2)} \Phi_1(r).
	\end{align*}
	Observe that $\Phi_1(r)$ is continuous on $[0,1]$ with
	\begin{equation*}
		\Phi_1(0) = A + 1 - 4\beta > 0 \quad \text{for } \beta < \frac{A+1}{4},
	\end{equation*}
	and
	\begin{equation*}
		\Phi_1(1) = (A + 4\beta - 3) - 2(A + 2\beta - 1) + A + 1 - 4\beta = -4\beta < 0.
	\end{equation*}
	By the \textit{Intermediate Value Theorem}, $\Phi_1(r) = 0$ has a unique root $R_{\beta,\mathrm{Co}(A)} \in (0,1)$. Thus, $\mathrm{Re}\,(T_f(z)) > 0$ for $|z| < R_{\beta,\mathrm{Co}(A)}$.\vspace{1.2mm}
	
	To establish sharpness, consider the function $f_\beta \in \mathcal{N}(\beta)$ defined by $f_\beta'(z) = (1-z)^{2(\beta-1)}$. A direct computation gives that
	\begin{equation*}
		1 + \frac{z f_\beta''(z)}{f_\beta'(z)} = 1 - \frac{2(\beta-1)z}{1-z}.
	\end{equation*}
	Choosing $z = -r \in (-1, 0)$, we find
	\begin{equation*}
		T_{f_\beta}(-r) = \frac{2}{A-1} \left( \frac{A+1}{2} \frac{1-r}{1+r} - 1 - \frac{2(\beta-1)r}{1+r} \right) = \frac{\Phi_1(r)}{(A-1)(1-r^2)}.
	\end{equation*}
	For $r > R_{\beta,\mathrm{Co}(A)}$, we have $\Phi_1(r) < 0$, which implies $\mathrm{Re}\,(T_{f_\beta}(-r)) < 0$. This completes the proof.
\end{proof}
We will first define the definition of radius of convexity and key conditions for convexity.
\begin{defi}
	The number $r\in [0, 1]$ is called the radius of convexity of a particular subclass $\mathcal{N}(\beta)$ of the class $\mathcal{A}$ of normalized analytic functions (where $f(z) = z + \sum_{n=2}^{\infty} a_n z^n$) in the unit disk $\mathbb{D}$ if $r$ is the largest number such that the function $f$ is convex in the disk $|z|<r$.
\end{defi}
A function $f$ analytic in a region $\Omega$ is convex in $\Omega$ if it maps $\Omega$ onto a convex region. For an analytic function $f$ in the unit disk $\mathbb{D}$, the condition for $f$ to be locally univalent and convex in a disk $|z| <r$ is given by the inequality you provided:
\begin{align}\label{Eq-5.2}
	{\rm Re}\left(1+\frac{zf^{\prime\prime}(z)} {f^{\prime}(z)}\right) > 0, \quad \text{for } |z| <r.
\end{align}
\begin{thm}\label{Th-5.2}
	Let $\beta > 1$. The radius of convexity $R_{\mathrm{c}}(\mathcal{N}(\beta))$ for the class  $\mathcal{N}(\beta)$ is 
	\begin{equation*}
		R_{\mathrm{c}}(\mathcal{N}(\beta)) = \frac{1}{2\beta - 1}.
	\end{equation*}
	The radius $R_{\mathrm{c}}(\mathcal{N}(\beta))$ is sharp.
\end{thm}
\begin{proof}[\bf Proof of Theorem \ref{Th-5.2}]
	Let $f \in \mathcal{N}(\beta)$ with $\beta > 1$. By Theorem 2.1(iii), we have
	\begin{equation*}
		\left| (1 - |z|^2) \frac{f''(z)}{f'(z)} + 2(\beta - 1)\bar{z} \right| \leq 2(\beta - 1).
	\end{equation*}
	This yields the sharp inequality
	\begin{equation*}
		\left| \frac{z f''(z)}{f'(z)} + \frac{2(\beta - 1)|z|^2}{1 - |z|^2} \right| \leq \frac{2(\beta - 1)|z|}{1 - |z|^2}.
	\end{equation*}
	For $|z| = r < 1$, taking the real part gives
	\begin{align*}
		\mathrm{Re}\left( 1 + \frac{z f''(z)}{f'(z)} \right) &\geq 1 - \frac{2(\beta - 1)r^2}{1 - r^2} - \frac{2(\beta - 1)r}{1 - r^2} \\
		&= 1 - \frac{2(\beta - 1)r}{1 - r} \\
		&= \frac{1 - (2\beta - 1)r}{1 - r}.
	\end{align*}
	Thus, $\mathrm{Re}\left( 1 + \frac{z f''(z)}{f'(z)} \right) > 0$ provided that $1 - (2\beta - 1)r > 0$, which holds for 
	\begin{equation*}
		r < \frac{1}{2\beta - 1}.
	\end{equation*}
	Therefore, $f(z)$ is convex in the disk $|z| < R_{\mathrm{c}}(\mathcal{N}(\beta)) = {1}/{(2\beta - 1)}$.\vspace{1.2mm}
	
	To show that this radius is sharp, consider the function $f_\beta \in \mathcal{N}(\beta)$ defined by 
	\begin{equation*}
		f_{\beta}'(z) = (1 - z)^{2(\beta - 1)}.
	\end{equation*}
	A direct computation gives
	\begin{equation*}
		1 + \frac{z f_\beta''(z)}{f_\beta'(z)} = 1 - \frac{2(\beta - 1)z}{1 - z}.
	\end{equation*}
	Setting $z = r \in (0,1)$, we obtain
	\begin{equation*}
		1 + \frac{r f_\beta''(r)}{f_\beta'(r)} = 1 - \frac{2(\beta - 1)r}{1 - r} = \frac{1 - (2\beta - 1)r}{1 - r}.
	\end{equation*}
	For any $r > {1}/{(2\beta - 1)}$, we have 
	\begin{align*}
		\mathrm{Re}\left(1 + \frac{r f_\beta''(r)}{f_\beta'(r)}\right) < 0.
	\end{align*} Hence, the function $f_\beta$ is not convex in any larger disk. This completes the proof.
\end{proof}

\noindent{\bf Acknowledgment:} We would like to thank the referee(s) for their thorough review and valuable suggestions, which will greatly improved the clarity and quality of the manuscript. 
\vspace{1.2mm}

%\noindent {\bf Funding:} Not Applicable.\vspace{1.2mm}

\noindent\textbf{Compliance of Ethical Standards:}\\

\noindent\textbf{Conflict of interest:} The authors declare that there is no conflict  of interest regarding the publication of this paper.\vspace{1.2mm}

\noindent\textbf{Data availability statement:}  Data sharing not applicable to this article as no datasets were generated or analysed during the current study.\vspace{1.2mm}

\noindent {\bf Authors' contributions:} All the three  authors have equal contributions in preparation of the manuscript.

\end{document}